\documentclass[10pt,a4paper]{article}

\usepackage[utf8]{inputenc}
\usepackage[english]{babel}

\usepackage{amsmath}
\usepackage{amsfonts}
\usepackage{amssymb}
\usepackage{amsthm}

\usepackage{graphicx}
\usepackage[margin=2cm]{geometry}

\usepackage{paralist}

\usepackage{cite}

\usepackage{hyperref}

\hypersetup{
    colorlinks=true,
    linkcolor=red,
    filecolor=magenta,      
    urlcolor=cyan
            } 
 
\usepackage{amsthm}

\usepackage{amsthm}

\usepackage{chngcntr}

\usepackage{caption}
\usepackage{subcaption}
\usepackage{wrapfig}

\newtheorem{theorem}{Theorem}
\newtheorem{lemma}{Lemma}

\newtheorem{example}{Example}
\newtheorem{remark}{Remark}

\counterwithin{theorem}{section}
\counterwithin{definition}{section}
\counterwithin{example}{section}
\counterwithin{lemma}{section}
\counterwithin{remark}{section}

\DeclareMathOperator{\ord}{Ord}

\graphicspath{{Pictures/}}

\begin{document}
\begin{center}
\huge{ Uniformly Convergent Difference Scheme for a Semilinear Reaction-Diffusion Problem on a Shishikin mesh}\\
\LARGE{Samir Karasulji\' c\footnote{corresponding author}, Enes Duvnjakovi\' c and Elvir Memi\' c}
\end{center}

\begin{abstract}
  In this paper we consider two difference schemes for numerical solving of a one--dimensional singularly perturbed boundary value problem.  We proved an $\varepsilon$--uniform convergence for both difference schemes on a Shiskin mesh. Finally, we present four numerical experiments  to confirm the theoretical results. 
\end{abstract}
\section{Introduction}

We consider the semilinear singularly perturbed problem
\begin{equation}
\varepsilon^2y''(x)=f(x,y)\:\text{ on }\:\left( 0,1\right),
\label{uvod1}
\end{equation}
\begin{equation}
y(0)=0,\:y(1)=0,
\label{uvod2}
\end{equation}
where $\varepsilon$ is a small positive parameter. We assume that the nonlinear function $f$ is continuously differentiable,  i.e.  for $k\geq 2,\:f\in C^k(\left[0,1\right]\times \mathbb{R}),$  and that it has a strictly positive derivative with respect to $y$
\begin{equation}
\dfrac{\partial f}{\partial y}=f_y\geq m>0\:\text{ on }\:\left[0,1\right]\times \mathbb{R}\:\:(m=const).
\label{uvod3}
\end{equation}
The boundary value problem \eqref{uvod1}--\eqref{uvod2}, under the condition \eqref{uvod3}, has a unique solution (see \cite{lorenz1982stability}).  Numerical treatment of the problem \eqref{uvod1}, has been considered by many authors, under  different condition on the function $f,$ and made a significant contribution. 

We are going to analyze two difference schemes for the problem \eqref{uvod1}--\eqref{uvod3}. These difference schemes were constructed using the method first introduced by Boglaev \cite{boglaev1984approximate}, who constructed a difference scheme and showed convergence of order 1 on a modified Bakhvalov mesh. In our previous papers using the method \cite{boglaev1984approximate}, we constructed new difference schemes in \cite{samir2010scheme, samir2011scheme, samir2011uniformnly, samir2011skoplje, samir2012uniformnly, samir2012class, samir2013collocation, samir2015construction} and performed numerical tests, in \cite{samir2015uniformlyconvergent, samir2015uniformly} we constructed new difference schemes and we proved the theorems on the uniqueness of the numerical solution and the $\varepsilon$--uniform convergence on the modified Shishkin mesh, and again performed the numerical test. In \cite{samir2017construction} we used the difference schemes from \cite{ samir2015uniformly} and calculated the values of the approximate solutions of the problem \eqref{uvod1}--\eqref{uvod3}  on the mesh points and then we constructed an approximate solution.

Since in the boundary layers, i.e. near $x=0$ and $x=1,$ the solution of the problem \eqref{uvod1}--\eqref{uvod3}  changes rapidly, when parameter tends to zero, in order to get the $\varepsilon$--uniform convergence, we have to use  a layer-adapted mesh. In the present paper we are going to use a Shishkin mesh \cite{shishkin1988grid}, which is piecewise equidistant and consequently simpler than the modified Shishkin mesh we have already used in our mentioned papers.

\section{Difference schemes}
For a given positive integer $N,$ let it be an arbitrary mesh
\[0=x_0<x_1<\cdots<x_{N-1}<x_N=1,\]
with $h_i=x_i-x_{i-1},$ for $i=1,\ldots,N.$
 
Our first difference scheme has the following form
 \begin{equation}
\left( \dfrac{a_i+d_i}{2}\right)\overline{y}_{i-1}-\left( \dfrac{a_i+d_i}{2}+\dfrac{a_{i+1}+d_{i+1}}{2}\right)\overline{y}_{i}+\left( \dfrac{a_{i+1}+d_{i+1}}{2}\right)\overline{y}_{i+1}=\dfrac{\triangle d_{i}}{\gamma}\overline{f}_{i-1/2}+\dfrac{\triangle d_{i+1}}{\gamma}\overline{f}_{i/2},
\label{dodatak17}
\end{equation}
where $d_i=\frac{\beta}{\tanh(\beta h_{i-1})},\,a_i=\frac{\beta}{\sinh(\beta h_{i-1})},\,\overline{f}_{i-1/2}=f\left( \frac{x_{i-1}+x_{i}}{2},\frac{\overline{y}_{i-1}+\overline{y}_i}{2}\right)$ and $\triangle d_i=d_i-a_i.$

From \eqref{dodatak17}, we obtain next discrete problem
\begin{equation}
F\overline{y}=\left(F\overline{y}_0,F\overline{y}_1,\ldots,F\overline{y}_N \right)^T,
 \label{diskretni}
\end{equation}
where
\begin{align}
 F\overline{y}_0&=\overline{y}_0=0,\nonumber\\  
  F\overline{y}_i&=  
            \dfrac{\gamma}{\triangle d_i+\triangle d_{i+1}}\left[ \dfrac{a_i+d_i}{2}\overline{y}_{i-1}-\left( \dfrac{a_i+d_i}{2}
                    +\dfrac{a_{i+1}+d_{i+1}}{2}\right)\overline{y}_{i}\right.\nonumber\\
             & \hspace{3.5cm}
                 +\left. \dfrac{a_{i+1}+d_{i+1}}{2}\overline{y}_{i+1}
                     -\dfrac{\triangle d_{i}}{\gamma}\overline{f}_{i-1/2}-\dfrac{\triangle d_{i+1}}{\gamma}\overline{f}_{i/2}\right]=0,\,i=1,2,\ldots,N-1,
                          \label{diskretni1}\\
             F\overline{y}_N&= \overline{y}_N=0,\nonumber
\end{align}
and $\overline{y}:=(\overline{y}_0,\overline{y}_1,\ldots,\overline{y}_N)^T$ is the solution of the problem
\begin{equation}
 F\overline{y}=0.
 \label{diskretni2}
\end{equation}

Second difference scheme  has the following form
\begin{multline} 
  (3a_{i}+d_{i}+\triangle d_{i+1})\left(\tilde{y}_{i-1}-\tilde{y}_i\right)
     -(3a_{i+1}+d_{i+1}+\triangle d_{i})\left(\tilde{y}_i-\tilde{y}_{i+1}\right)  
=\dfrac{\tilde{f}_{i-1}+2\tilde{f}_{i}+\tilde{f}_{i+1} }{\gamma}\left(\triangle d_{i}+\triangle d_{i+1}\right),
\label{schema2}
\end{multline}
where $d_i=\frac{\beta}{\tanh(\beta h_{i-1})},\,a_i=\frac{\beta}{\sinh(\beta h_{i-1})},\,\tilde{f}_{i}=f\left( x_{i},\tilde{y}_{i}\right)$ and $\triangle d_i=d_i-a_i.$ 

From \eqref{schema2}, we obtain second discrete problem
 \begin{equation}
    G\tilde{y}=(G\tilde{y}_0,G\tilde{y}_1,\ldots,G\tilde{y}_N)^T,
  \label{scheme2a}
 \end{equation}
where 
\begin{align}
  G\tilde{y}_0&=\tilde{y}_0=0\\
  G\tilde{y}_i&= \frac{\gamma}{\triangle d_i+\triangle d_{i+1}}\biggl[(3a_{i}+d_{i}+\triangle d_{i+1})\left(\tilde{y}_{i-1}-\tilde{y}_i\right)
                              -(3a_{i+1}+d_{i+1}+\triangle d_{i})\left(\tilde{y}_i-\tilde{y}_{i+1}\right) \\ 
                    &\hspace{3cm}-\left.\dfrac{\tilde{f}_{i-1}+2\tilde{f}_{i}+\tilde{f}_{i+1}}
                                       {\gamma}\left(\triangle d_{i}+\triangle d_{i+1}\right)\right]=0,\;i=1,\ldots,N-1,\label{schema2aa}\\
  G\tilde{y}_N&=\tilde{y}_N=0,                                      
\end{align}
and $\tilde{y}=(\tilde{y}_0,\tilde{y}_1,\ldots,\tilde{y}_n)^T$ is the solution of the problem
\begin{equation}
       G\tilde{y}=0.
  \label{schema2b}
\end{equation}

\section{Theoretical background}

In this paper we use the maximum norm
 \begin{equation}
      \left\|u\right\|_{\infty}=\max_{0\leqslant i\leqslant N}\left|u_i\right|,
   \label{norm}
 \end{equation}
for any vector $u=\left(u_0,u_1,\ldots,u_n\right)^T\in\mathbb{R}^{N+1}$ and the corresponding matrix norm.

The next two theorems hold

\begin{theorem} {\rm\cite{samir2015uniformly}} The discrete problem \eqref{diskretni2} for $\gamma\geq f_y,$ has  the unique solution\\ $\overline{y}=(\overline{y}_0,\overline{y}_1,\overline{y}_2,\ldots,\overline{y}_{N-1},\overline{y}_{N})^{T},$ with $\overline{y}_0=\overline{y}_N=0.$ Moreover, the following stability inequality holds 
\begin{equation}
     \left\|w-v\right\|_{\infty}\leqslant \frac{1}{m}\left\|Fw-Fv\right\|_{\infty},
 \label{stabilnost}
\end{equation}
for any vectors $v=\left(v_0,v_1,\ldots,v_N\right)^T\in\mathbb{R}^{N+1},\,w=\left( w_0,w_1,\ldots,w_N\right)^T\in\mathbb{R}^{N+1}.$
\end{theorem}

\begin{theorem} {\rm\cite{samir2015uniformlyconvergent}} The discrete problem \eqref{schema2b} has a unique solution $\tilde{y}$ for $\gamma\geqslant f_y$. Also, for every $u,v\in\mathbb{R}^{N+1}$ we have the following stabilizing inequality
\begin{equation*}
    \left\|u-v\right\|_{\infty}\leqslant \frac{1}{m}\left\|Gu-Gv\right\|_{\infty}.
\end{equation*}
\end{theorem}

In the following analysis we need the decomposition of the solution $y$ of the problem 
$(\ref{uvod1})-(\ref{uvod2})$ to the layer component $s$ and a regular component $r$, given in the following assertion. 

\begin{theorem} {\rm\cite{vulanovic1983} 
} The solution $y$ to problem $(\ref{uvod1})-(\ref{uvod2})$ can be represented in the following way:
\begin{equation*}
   y=r+s,
\end{equation*}
where for $j=0,1,...,k+2$ and $x\in[0,1]$ we have  that
\begin{equation}
\left|r^{(j)}(x)\right|\leqslant C,
\label{regularna}
\end{equation}
and
\begin{equation}
\left|s^{(j)}(x)\right|\leqslant C \varepsilon^{-j}\left(e^{-\frac{x}{\varepsilon}\sqrt{m}}+e^{-\frac{1-x}{\varepsilon}\sqrt{m}}\right).
\label{slojna}
\end{equation}
\label{teorema1}
\end{theorem}

\label{mreza}
\section{Construction of the mesh}
The solution of the problem \eqref{uvod1}--\eqref{uvod3} changes fast near the ends of our domain $[0,1].$ Therefore, the mesh has to be refined there. A Shishin mesh is used to resolve the layers. This mesh is piecewise equidistant and it's  quite simple. It is constructed as follows (see \cite{stynes1996}). For given a positive integer $N,$ where $N$ is divisible by 4, we divide the interval $[0,1]$ into three subintervals 
\[[0,\lambda],\quad[\lambda,1-\lambda],\quad[1-\lambda,1].\]
We use equidistant meshes on each of these subintervals, with $1+\frac{N}{4}$ points in each of $[0,\lambda]$ and $[1-\lambda,1],$ and $1+\frac{N}{2}$ points in 
$[1-\lambda,1].$ We define the parameter $\lambda$ by
\[\lambda=\min\left\{ \frac{1}{4},\frac{2\varepsilon\ln N}{\sqrt{m}}\right\},\]
which depends on $N$ and $\varepsilon.$ The basic idea here is to use a fine mesh to resolve the part of the boundary layers. More precisely, we have
\[0=x_0<x_1<\ldots<x_{i_0}<\ldots<x_{N-i_0}<\ldots<x_{N-1}<x_N=1,\]
with $i_0=N/4,\:x_{i_0}=\lambda,\:x_{N-i_0}=1-\lambda,$ and 
\begin{align}
 h_{i-1}&=\frac{4\lambda }{N}\text{ for }i=1,\ldots,i_0,N-i_0,\ldots,N,\label{mesh1}\\
 h_{i-1}&=\frac{2(1-2\lambda)}{N}\text{ for } i=i_0+1,\ldots,N-i_0.\label{mesh2}
\end{align}
If $\lambda=\frac{1}{4}$ i.e. $\frac{1}{4}\leqslant\frac{2\varepsilon\ln N}{N},$ then $\frac{1}{N}$ is very small relative to $\varepsilon.$ This is unlike in practice, and in this case  the method can be analyzed using standard techniques. Hence, we assume that
\begin{equation}
   \lambda=\frac{2\varepsilon\ln N}{\sqrt{m}}.
 \label{mesh3}
\end{equation}
From \eqref{mesh1} and \eqref{mesh2}, we conclude  that that the interval lengths satisfy
\begin{equation}
   h_{i-1}=\frac{8\varepsilon\ln N}{\sqrt{m}}\text{ for } i=1,\ldots,i_0,N-i_0,\ldots,N,
\label{mesh4}
\end{equation} 
and 
\begin{equation}
  \frac{1}{N}\leqslant h_{i-1}\leqslant\frac{2}{N}\text{ for }i=i_0+1,\ldots,N-i_0.
\label{mesh5}
\end{equation}
 
\section{Uniform convergence}

We will prove the theorem on uniform convergence of the difference schemes \eqref{dodatak17} and \eqref{schema2}   on the part of the mesh which corresponds to  $[0,1/2],$ while the proof on $[1/2,1]$ can be analogously derived.

Namely, in the analysis of the value of the error the functions $e^{-\frac{x}{\varepsilon}\sqrt{m}}$ and $e^{-\frac{1-x}{\varepsilon}\sqrt{m}}$ appear. For these functions we have that $e^{-\frac{x}{\varepsilon}\sqrt{m}}\geqslant e^{-\frac{1-x}{\varepsilon}\sqrt{m}},\:\forall x\in [0,1/2]$ and $e^{-\frac{x}{\varepsilon}\sqrt{m}}\leqslant e^{-\frac{1-x}{\varepsilon}\sqrt{m}},\:\forall x\in [1/2,1]$. In the boundary layer in the neighbourhood of $x=0$, we have that $e^{-\frac{x}{\varepsilon}\sqrt{m}}>> e^{-\frac{1-x}{\varepsilon}\sqrt{m}}$, while in the boundary layer in the neighbourhood of $x=1$ we have that $e^{-\frac{x}{\varepsilon}\sqrt{m}}<< e^{-\frac{1-x}{\varepsilon}\sqrt{m}}.$  Based on the above, it is enough to prove the theorem on the part of the mesh which corresponds to  $[0,1/2]$ with the exclusion of the function $e^{-\frac{1-x}{\varepsilon}\sqrt{m}}$, or on $[1/2,1]$ but with the exclusion of the function $e^{-\frac{x}{\varepsilon}\sqrt{m}}$. Note that we need to take care of the fact that in the first case $h_{i-1}\leqslant h_i,$ and in the second case $h_{i-1}\geqslant h_i$.

Let us start with the following two lemmas that will be further used in the proof of the first uniform convergence theorem on the part of the mesh from Section \ref{mreza} which corresponds to  $\left[x_{N/4-1},1/2\right]$ and $x_{N/4}=\lambda.$

\begin{lemma}\label{lema1}
Assume that $\varepsilon\leqslant \frac{C}{N}.$ In the part of the Shiskin mesh from Section {\rm\ref{mreza}}, when $x_i,\,x_{i\pm 1}\in[x_{N/4},1/2],$ we have the following estimate
\begin{equation}
    \left|Fy_i\right|\leqslant \frac{C}{N^2},\:i=N/4,\ldots,N/2-1.
\end{equation} 
\end{lemma}

\begin{proof}
On this part of the mesh holds $h_{i-1}=h_i,$ so we have that 
\begin{align*}
Fy_i&=\frac{\gamma}{2(\cosh(\beta h_{i})-1)}\left[(1+\cosh(\beta h_i)(y_{i-1}-2y_i+y_{i+1})-\frac{\cosh(\beta h_i)-1}{\gamma}(f_{i-1/2}+f_{i/2}) \right]\\
  &=\frac{\gamma}{2}\left(y_{i-1}-2y_i+y_{i+1}-\frac{f_{i-1/2}+f_{i/2}}{\gamma} \right)-\frac{\gamma}{\cosh(\beta h_i)-1}\left(y_{i-1}-2y_i+y_{i+1} \right).
\end{align*}

Because of Theorem \ref{teorema1}, and the fact that $\varepsilon^2y''=f(x,y),\,x\in(0,1),$ we obtain

\begin{multline*}
\left| Fy_i\right|\leqslant C_1 \biggl[ \left| r_{i-1}-2r_i+r_{i+1}\right|+\left|s_{i-1}-2s_i+s_{i+1} \right| +\varepsilon^2|y''_{i-1}|\\
       \left.+\frac{1}{\cosh(\beta h_i)-1}\left(\left| r_{i-1}-2r_i+r_{i+1}\right|+\left|s_{i-1}-2s_i+s_{i+1} \right|\right) \right].
\end{multline*}

Again, due to Theorem \ref{teorema1} and Taylor expansion, the following inequalities hold 
\begin{align*}
   \left| r_{i-1}-2r_i+r_{i+1}\right|&=\left|\frac{r''(\xi^{-}_i)}{2}h^2_i+\frac{r''(\xi^{+}_i)}{2}h^2_i \right|   \leqslant C_2 h^2_i,\\
   \left|s_{i-1}-2s_i+s_{i+1} \right|&\leqslant \frac{C_3}{N^2},\\
   \frac{1}{\cosh(\beta h_i)-1}&\leqslant\frac{2}{(\beta h_i)^2}=\frac{2\varepsilon^2}{\gamma h^2_i}\leqslant C_4,\\
   \varepsilon^2|y''_{i-1}|&\leqslant C_5\varepsilon^2\left|\varepsilon^{-2}(e^{-\frac{x_{i-1}}{\varepsilon}\sqrt{m}}+e^{-\frac{1-x_{i-1}}{\varepsilon}\sqrt{m}})+r''_{i-1}\right|
        \leqslant C_6\left(\frac{1}{N^2}+\varepsilon^2\right),
\end{align*}
where $\xi_i^{-}\in(x_{i-1},x_i)$ and $\xi_i^{+}\in(x_i,x_{i+1}).$
Finally, we have that 
\begin{equation}
  \left|Fy_i\right|\leqslant \frac{C}{N^2}.
\end{equation}

\end{proof}

\begin{lemma}\label{lema2}
Assume that $\varepsilon\leqslant \frac{C}{N}.$ In the part of the Shiskin mesh from Section {\rm\ref{mreza}}, when $x_i=x_{N/4},$ we have the following estimate
\begin{equation}
    \left|Fy_{N/4}\right|\leqslant \frac{C}{N}.
    \label{lema2a}
\end{equation} 
  
\end{lemma}

\begin{proof}
 Let us estimate $\left\|Fy_{N/4}\right\|_{\infty},$  consider $Fy_i$ in the following form
 \begin{multline}
          Fy_i=\frac{\gamma}{\dfrac{\cosh(\beta h_{i-1})-1}{\sinh(\beta h_{i-1})}+\dfrac{\cosh(\beta h_i)-1}{\sinh(\beta h_i)}}
             \left[ \frac{1+\cosh(\beta h_{i-1})}{2\sinh(\beta h_{i-1})}y_{i-1}
             -\left(\frac{1+\cosh(\beta h_{i-1})}{2\sinh(\beta h_{i-1})} +\frac{1+\cosh(\beta h_{i})}{2\sinh(\beta h_{i})}\right)y_i\right.\\
               \left.+\frac{1+\cosh(\beta h_{i})}{2\sinh(\beta h_{i})}y_{i+1}
               -\frac{\cosh(\beta h_{i-1})-1}{\gamma \sinh(\beta h_{i-1})}f_{i-1/2}-\frac{\cosh(\beta h_{i})-1}{\gamma \sinh(\beta h_{i})}f_{i/2}\right],\:i=N/4
   \label{lema2b}              
 \end{multline} 

Let us first estimate the expressions from \eqref{lema2b} using the nonlinear terms. Due to Theorem \ref{teorema1}, and the fact that $\varepsilon^2y''=f(x,y),\,x\in(0,1),$  we have that 

\begin{align}
   \frac{\gamma}{\dfrac{\cosh(\beta h_{i-1})-1}{\sinh(\beta h_{i-1})}+\dfrac{\cosh(\beta h_i)-1}{\sinh(\beta h_i)}}
     \left| -\frac{\cosh(\beta h_{i-1})-1}{\gamma \sinh(\beta h_{i-1})}f_{i-1/2}-\frac{\cosh(\beta h_{i})-1}{\gamma \sinh(\beta h_{i})}f_{i/2}\right|
       \leqslant& C_3 \varepsilon^2y''(x_{N/4})\leqslant\frac{C_4}{N^2}.
 \label{lema2c}
\end{align}

For the linear terms from \eqref{lema2b}, we have that

\begin{align}
 & \frac{\gamma}{\dfrac{\cosh(\beta h_{i-1})-1}{\sinh(\beta h_{i-1})}+\dfrac{\cosh(\beta h_i)-1}{\sinh(\beta h_i)}}\\
 &   \hspace{2cm}  \cdot\left[ \frac{1+\cosh(\beta h_{i-1})}{2\sinh(\beta h_{i-1})}y_{i-1}
             -\left(\frac{1+\cosh(\beta h_{i-1})}{2\sinh(\beta h_{i-1})} +\frac{1+\cosh(\beta h_{i})}{2\sinh(\beta h_{i})}\right)y_i
              +\frac{1+\cosh(\beta h_{i})}{2\sinh(\beta h_{i})}y_{i+1}\right]\\
 &=  \frac{\gamma}{\dfrac{\cosh(\beta h_{i-1})-1}{\sinh(\beta h_{i-1})}+\dfrac{\cosh(\beta h_i)-1}{\sinh(\beta h_i)}}
   \left[ \frac{1+\cosh(\beta h_{i-1})}{2\sinh(\beta h_{i-1})}(y_{i-1}-y_i)-\frac{1+\cosh(\beta h_{i})}{2\sinh(\beta h_{i})}(y_i-y_{i+1})      \right].           
\end{align}

According Theorem \ref{teorema1}, for the layer component $s,$ we have that 

\begin{align}
  \frac{\gamma}{\dfrac{\cosh(\beta h_{i-1})-1}{\sinh(\beta h_{i-1})}+\dfrac{\cosh(\beta h_i)-1}{\sinh(\beta h_i)}}
   &\left| \frac{1+\cosh(\beta h_{i-1})}{2\sinh(\beta h_{i-1})}(s_{i-1}-s_i)-\frac{1+\cosh(\beta h_{i})}{2\sinh(\beta h_{i})}(s_i-s_{i+1})\right|\nonumber\\   
   &\leqslant C_5\left( |s_{i-1}-s_i| +|s_i-s_{i+1}|\right)\leqslant\frac{C_6}{N^2}.  \label{lema2d}
\end{align}

For the regular component $r,$ due to $\frac{\cosh x-1}{\sinh x}=\tanh\frac{x}{2}$  and our assumption $\varepsilon\leqslant 1/N,$ we get that

\begin{align}
&\frac{\gamma}{\dfrac{\cosh(\beta h_{i-1})-1}{\sinh(\beta h_{i-1})}+\dfrac{\cosh(\beta h_i)-1}{\sinh(\beta h_i)}}
   \left| \frac{1+\cosh(\beta h_{i-1})}{2\sinh(\beta h_{i-1})}(r_{i-1}-r_i)-\frac{1+\cosh(\beta h_{i})}{2\sinh(\beta h_{i})}(r_i-r_{i+1})\right|\nonumber\\
&= \frac{\gamma}{\tanh\frac{\beta h_{i-1}}{2}+\tanh\frac{\beta h_i}{2}}
            \left|\frac{\tanh\frac{\beta h_{i-1}}{2}}{2}(r_{i-1}-r_i)+\frac{\tanh\frac{\beta h_{i}}{2}}{2}(r_{i}-r_{i+1})+2(r_{i-1}-r_i)-2(r_i-r_{i+1}) \right|\nonumber\\
&\leqslant C_7 \left(  |r_{i-1}-r_i|+|r_i-r_{i-1}|+\frac{|r_{i-1}-r_i|+|r_i-r_{i+1}|}{\tanh(\beta h_i)} \right)  \nonumber\\
&\leqslant C_8\left( \frac{\varepsilon\ln N}{N}+\frac{1}{N}+\frac{\frac{\varepsilon\ln N}{N}+\frac{1}{N}}{\tanh(\beta h_i)}\right)   \leqslant \frac{C}{N}.  
 \label{lema2e}       
\end{align}

Now, collecting   \eqref{lema2c}, \eqref{lema2d} and \eqref{lema2e},  the statement of the lemma is therefore proven. 
\end{proof}

\begin{theorem}
The discrete problem \eqref{diskretni2} on the mesh from Section  {\rm\ref{mreza}} is uniformly convergent with respect to $\varepsilon$ and
   \begin{equation}
     \max_{i}\left|y_i-\overline{y}_i \right|\leqslant C
        \left\{        
        \begin{array}{cl}
                   \dfrac{\ln^2N}{N^2},\:& i\in\{0,1,\ldots,N/4-1\}\vspace{.25cm}\\
                   \dfrac{1}{N^2},\:& i\in\{N/4+1,\ldots,3N/4-1\}\vspace{.25cm}\\
                   \dfrac{1}{N},\:& i\in\{N/4,3N/4\}\vspace{.25cm}\\
                   \dfrac{\ln^2N}{N^2},\:& i\in\{3N/4+1,\ldots,N\}, 
        \end{array}
       \right. 
       \label{th1tez}
   \end{equation}
where $y$ is the solution of the problem \eqref{uvod1}--\eqref{uvod3}, $\overline{y}$ is the corresponding solution of \eqref{diskretni2} and $C>0$ is a constant independent of $N$ and $\varepsilon.$
\end{theorem}

\begin{proof}
 We are going to divide the proof of this theorem in four parts. 

Suppose first that $x_i,x_{i\pm 1}\in[0,\lambda],\:i=1,\ldots,N/4.$ The proof for this part of the mesh has already been done in \cite[Theorem 4.2]{samir2015uniformly}. It is hold that 
 \begin{equation}
     \left| Fy_i\right|\leqslant C\frac{\ln^2 N}{N^2},\:i=0,1,\ldots,N/4-1.
  \label{th1}
 \end{equation}

Now, suppose that $x_i,x_{i\pm i}\in[x_{N/4+1},x_{N/2-1}].$ Based on Lemma  \ref{lema1}, we have that
\begin{equation}
    \left|Fy_i\right|\leqslant \frac{C}{N^2}.
  \label{procjenath2}
\end{equation}  

In the case $i=N/4,$ now based on Lemma \ref{lema2}, we have that

 \begin{equation}
    \left|Fy_{N/4}\right|\leqslant \frac{C}{N}.
  \label{th3}
\end{equation}

Finally, the proof in the case $i=N/2$ is trivial, because the mesh on this part is equidistant and the influence of the layer component is  negligible. Therefore
\begin{equation}
   \left|Fy_{N/2}\right|\leqslant\frac{C}{N^2}.
\label{th4}
\end{equation}

Using inequalities \eqref{th1}, \eqref{procjenath2}, \eqref{th3} and \eqref{th4}, we complete the proof of the theorem.
\end{proof}

Let us show the $\varepsilon$--uniform convergence of second difference scheme, i.e \eqref{schema2}. 

\begin{lemma}\label{lema3}
   Assume that $\varepsilon\leqslant \frac{C}{N}.$ In the part of the Shiskin mesh from Section {\rm\ref{mreza}}, when $x_i,\,x_{i\pm 1}\in[x_{N/4},1/2],$ we have the following estimate
\begin{equation}
    \left|G y_i\right|\leqslant \frac{C}{N^2},\:i=N/4,\ldots,N/2-1.
     \label{lema3aa}
\end{equation} 
\end{lemma}

\begin{proof}
   Let us rewrite $G\tilde{y}_i$ in the following form
\begin{align}
  G\tilde{y}_i&=\frac{\gamma}{2\dfrac{\cosh(\beta h_i)-1}{\sinh(\beta h_i)}}
   \left[ \frac{2(\cosh(\beta h_i)+1)}{\sinh(\beta h_i)}(y_{i-1}-y_{i})-\frac{2(\cosh(\beta h_i)+1)}{\sinh(\beta h_i)}(y_i-y_{i+1}) \right.\nonumber\\
           &\hspace{3cm}-\left.\varepsilon^2\frac{y''_{i-1}-y''_i+y''_{i+1}}{\gamma}\cdot\frac{2(\cosh(\beta h_i)-1)}{\sinh(\beta h_i)}\right]\nonumber\\
           &=\frac{\gamma}{\cosh(\beta h_i)-1}\left[ (\cosh(\beta h_i)-1)(y_{i-1}-2y_i+y_{i+1})-2(y_{i-1}-2y_i+y_{i+1})\right.\nonumber\\
           &\hspace{3cm}   -\left.\varepsilon^2(y''_{i-1}-2y''_i+y''_{i+1})\cdot\frac{\cosh(\beta h_i)-1}{\gamma}\right]\nonumber\\
           &=\gamma(y_{i-1}-2y_i+y_{i+1})-\frac{2\gamma(y_{i-1}-2y_i+y_{i+1})}{\cosh(\beta h_i)-1} -\varepsilon^2(y''_{i-1}-y''_i+y''_{i+1}).\label{lema3a}  
\end{align} 
Using Theorem \ref{teorema1}, Taylor expansion, assumption $\varepsilon\leqslant\frac{1}{N}$ and the properties of the mesh from Section \ref{mreza}, let us estimate the expressions from \eqref{lema3a}.  We get that

\begin{align}
 |y_{i-1}-2y_i+y_{i+1}|&\leqslant C_1\left( |r_{i-1}-r_i+r_{i+1}|+|s_{i-1}-2s_i+s_{i+1}|\right)\nonumber\\
                       &\leqslant C_2\left(\frac{\left|r''(\xi^{+}_i)+r''(\xi^{-}_i)\right|}{2}h^2_i+e^{-\frac{x_{i-1}}{\varepsilon}\sqrt{m}} \right)
                         \leqslant \frac{C_3}{N^2},\label{lema3b}\\
  \frac{1}{\cosh(\beta h_i)-1}&\leqslant\frac{2}{(\beta h_i)^2} =\frac{2\varepsilon^2}{\gamma h^2_i}\leqslant C_3,\label{lema3c}\\  
  \varepsilon^2\left| y''_{i-1}-y''_i+y''_{i+1}\right|&\leqslant \varepsilon^2\left| r''_{i-1}-r''_i+r''_{i+1}\right|
                               +\varepsilon^2\left|s''_{i-1}-s''_i+s''_{i+1}\right| \nonumber\\
                       &\leqslant C_4\varepsilon^2\left( 1+\frac{e^{-\frac{x_{i-1}}{\varepsilon}\sqrt{m}}}{\varepsilon^2}\right)\leqslant \frac{C_5}{N^2},\label{lema3d}
\end{align}   
where $\xi^{-}_{i}\in(x_{i-1},x_i),\:\xi^{+}_i\in(x_i,x_{i+1}).$

Now using \eqref{lema3a}, \eqref{lema3b},\eqref{lema3c} and \eqref{lema3d}, we obtain \eqref{lema3aa}.

\end{proof}

\begin{lemma}\label{lema4}
  Assume that $\varepsilon\leqslant \frac{C}{N}.$ In the part of the Shiskin mesh from Section {\rm\ref{mreza}}, when $x_i=x_{N/4},$ we have the following estimate
\begin{equation}
    \left|Gy_{N/4}\right|\leqslant \frac{C}{N}.
    \label{lema4a}
\end{equation}
\end{lemma}

\begin{proof}
Using \eqref{schema2aa}, let us write $Gy_i$ in the following form

\begin{align}
  Gy_i&=\frac{\gamma}{\triangle d_i+\triangle d_{i+1}}
     \left[ (4a_i+\triangle d_i+\triangle d_{i+1})(y_{i-1}-y_i) -(4a_{i+1}+\triangle d_i+\triangle d_{i+1})(y_{i}-y_{i+1})\right] \nonumber\\
      &\hspace{.5cm}-\left(f_{i-1}+2f_i+f_{i+1} \right)\nonumber\\
      &=\frac{4\gamma}{\triangle d_i+\triangle d_{i+1}}\left[a_i(y_{i-1}-y_i)-a_{i+1}(y_i-y_{i+1}) \right]+\gamma(y_{i-1}-2y_i+y_{i+1})-\left(f_{i-1}+2f_i+f_{i+1} \right).
       \label{lema4b}   
\end{align}

In a similar way, as in the previously lemmas, we can get 
\begin{align}
  |y_{i-1}-2y_i+y_{i+1}|&\leqslant|s_{i-1}-2s_i+s_{i+1}|+|r_{i-1}-2r_i+r_{i+1}|\leqslant C_1\left( \frac{1}{N^2}+\frac{1}{N} \right)\label{lema4c},\\
  \left|f_{i-1}+2f_i+f_{i+1} \right|&\leqslant \frac{C_2}{N^2}.\label{lema4d}
\end{align}

Using the identity $\frac{\cosh x-1}{\sinh}=\tanh\frac{x}{2} $ and  Theorem \ref{teorema1}, we have that

\begin{align}
& \frac{4\gamma}{\triangle d_i+\triangle d_{i+1}}\left[ a_i(y_{i-1}-y_i)-a_{i+1}(y_i-y_{i+1}) \right]\nonumber\\
&\hspace{1cm}   =\frac{4\gamma}{\tanh\frac{\beta h_{i-1}}{2}+\tanh\frac{\beta h_i}{2}}
       \left[\frac{1}{\sinh(\beta h_{i-1})}|s_{i-1}-s_i|-\frac{1}{\sinh(\beta h_{i})}|s_{i}-s_{i+1}|\right.\nonumber\\
&\hspace{1.5cm}+\left.\frac{1}{\sinh(\beta h_{i-1})}|r_{i-1}-r_i|-\frac{1}{\sinh(\beta h_{i})}|r_{i}-r_{i+1}|  \right].\label{lema4e}
\end{align}

Due to Theorem \ref{teorema1} and assumption $\varepsilon\leqslant \frac{C}{N},$ hold the next inequalities
 \begin{align}
    \frac{\gamma}{\tanh\frac{\beta h_{i-1}}{2}+\tanh\frac{\beta h_i}{2}}&\leqslant\frac{4\gamma}{\tanh\frac{\beta h_i}{2}}\leqslant C_1,\label{lema4f}\\
    \frac{1}{\sinh(\beta h_{i-1})}|s_{i-1}-s_i|&\leqslant \frac{1}{\beta h_{i-1}}|s_{i-1}-s_i|
           \leqslant C_2\cdot\frac{1}{\frac{\ln N}{N}}\cdot\frac{1}{N^2}=\frac{C_2}{N\ln N},\label{lema4g}\\
    \frac{1}{\sinh(\beta h_i)}|s_{i}-s_{i+1}|&\leqslant\frac{1}{\beta h_i}|s_{i}-s_{i+1}|\leqslant\frac{C_3}{N^2},\label{lema4h}\\
    \frac{1}{\sinh(\beta h_{i-1})}|r_{i-1}-r_i|&\leqslant\frac{1}{\beta h_{i-1}}|r_{i-1}-r_i|\leqslant\frac{1}{\frac{\ln N }{N}}
           \cdot C_4\frac{\varepsilon \ln N}{N}=C_4 \varepsilon,\label{lema4i}\\
    \frac{1}{\sinh(\beta h_{i})}|r_{i}-r_{i+1}|&\leqslant\frac{1}{\beta h_i}  |r_{i}-r_{i+1}|\leqslant\frac{C_5}{N}.    \label{lema4j} 
     \end{align}
Now, using \eqref{lema4b}, \eqref{lema4c}, \eqref{lema4d}, \eqref{lema4e}, \eqref{lema4f}, \eqref{lema4g}, \eqref{lema4h}, \eqref{lema4i} and  \eqref{lema4j}, we obtain \eqref{lema4a}.  
\end{proof}

\begin{theorem}\label{th2}
The discrete problem \eqref{scheme2a} on the mesh from Section  {\rm\ref{mreza}} is uniformly convergent with respect to $\varepsilon$ and
   \begin{equation}
     \max_{i}\left|y_i-\tilde{y}_i \right|\leqslant C
        \left\{        
        \begin{array}{cl}
                   \dfrac{\ln^2N}{N^2},\:& i\in\{0,1,\ldots,N/4-1\}\vspace{.25cm}\\
                   \dfrac{1}{N^2},\:& i\in\{N/4+1,\ldots,3N/4-1\}\vspace{.25cm}\\
                   \dfrac{1}{N},\:& i\in\{N/4,3N/4\}\vspace{.25cm}\\
                   \dfrac{\ln^2N}{N^2},\:& i\in\{3N/4+1,\ldots,N\}, 
        \end{array}
       \right. 
       \label{th2a}
   \end{equation}
where $y$ is the solution of the problem \eqref{uvod1}--\eqref{uvod3}, $\tilde{y}$ is the corresponding solution of \eqref{schema2b} and $C>0$ is a constant independent of $N$ and $\varepsilon.$

\end{theorem}

\begin{proof}
Again, let us divide the proof on four parts. \\
Suppose first that $x_i,\,x_{i\pm 1}\in[0,\lambda],\:i=1,\ldots,N/4.$ The proof for this part of the mesh has already been done in \cite[Theorem 4.4]{samir2015uniformlyconvergent}. It is proved that
\begin{equation}
   \left| Gy_i\right|\leqslant C\frac{\ln^2N}{N^2},\:i=0,1,\ldots,N/4-1.
 \label{th2b}
\end{equation}
Secondly, suppose that $x_i,\,x_{i\pm 1}\in[x_{N/4+1},x_{N/2}-1].$ Due to Lemma \ref{lema3}, we have that 
\begin{equation}
   \left| Gy_i\right|\leqslant \frac{C}{N^2}.
 \label{th2c}
\end{equation}
In the case $i=N/4,$ based on Lemma \ref{lema4}, we have the following estimate
\begin{equation}
   \left| Gy_{N/4}\right|\leqslant\frac{C}{N}.
 \label{th2d}
\end{equation}
At the end, in the case $i=N/2,$ the proof is trivial, because of the properties of the mesh and the layer component. Hence, it is true that
\begin{equation}
  \left|Gy_i\right|\leqslant\frac{C}{N^2}.
 \label{th2e}
\end{equation} 
Using \eqref{th2b}, \eqref{th2c}, \eqref{th2d} and \eqref{th2e}, we complete the statement of the theorem. 
\end{proof}

\section{Numerical experiments}

In this section we present numerical results to confirm the uniform accuracy of the discrete problems  \eqref{diskretni2} and \eqref{schema2b}. Both discrete problems will be checked on two different examples. First one is the linear boundary value problem, whose exact solution is known. Second example is the nonlinear boundary value problem whose exact solution is unknown.

For the problems from our examples whose exact solution is known, we calculate $E_N$ as
\begin{equation}
  E_N=\max_{0\leqslant i\leqslant N}\left|y(x_i)-\overline{y}^{N}(x_i)\right| \text{ or } 
   E_N=\max_{0\leqslant i\leqslant N}\left|y(x_i)-\tilde{y}^{N}(x_i)\right|, 
 \label{greska1}
\end{equation} 
for the problems, whose exact solution is unknown, we calculate $E_N$, as
\begin{equation}
   E_N=\max_{0\leqslant i\leqslant N}\left|\overline{y}^{2N}_S(x_i)-\overline{y}^{N}(x_i)\right| \text{ or } 
   E_N=\max_{0\leqslant i\leqslant N}\left|\tilde{y}^{2N}_S(x_i)-\tilde{y}^{N}(x_i)\right|, 
 \label{greska2}
\end{equation}
the rate of convergence  $\ord$ we calculate in the usual way
\begin{equation}
    \ord=\frac{\ln E_N-\ln E_{2N}}{\ln \frac{2k}{k+1}}
   \label{greska3}
\end{equation}
where $N=2^k,\:k=6,7,\ldots,11,\:$ $\overline{y}^N(x_i),\;\tilde{y}^N(x_i)$  are the values of the numerical solutions on a mesh with $N+1$ mesh points, and 
$\overline{y}^{2N}_S(x_i),\;\tilde{y}^{2N}_S(x_i)$  are  the values of the numerical solutions on a mesh with $2N+1$ mesh points and the transition points altered slightly to 
$\lambda_S=\min\left\{ \frac{2}{4},\frac{2\varepsilon}{\sqrt{m}}\ln\frac{N}{2}\right\}.$  

\begin{remark}
  In a case when the exact solution is unknown we  use the double mesh method, see {\rm \cite{doolan1980uniform, stynes1996, surla2003uniformly} for details}. 
\end{remark}

\begin{example} \upshape
 Consider the following problem
\begin{equation*}
 \epsilon^2y''=y+1-2\varepsilon^2+x(x-1)\quad\text{for }x\in(0,1),\quad y(0)=y(1)=0.
\end{equation*}
The exact solution of this problem is given by
$\displaystyle
  y(x)=\frac{e^{-\frac{x}{\epsilon}}+e^{-\frac{1-x}{\epsilon}}}{1+e^{-\frac{1}{\epsilon}}}-x(x-1)-1.
$
\label{primjer1}
The nonlinear system was solved using the initial condition $y_0=-0.5$ and the value of the constant  $\gamma=1$.
\end{example}

\begin{center}
\begin{table}[!h]\small
\centering
\begin{tabular}{c|cc|cc|cc}\hline
     $N$ &$E_n$&Ord&$E_n$&Ord&$E_n$&Ord\\\hline
$2^{6}$&$8.1585e-04$&$2.00$   &$2.8932e-03$&$2.02$    &$2.5827e-02$&$2.05$ \\
$2^{7}$&$2.7762e-04$&$2.00$   &$9.7397e-04$&$2.01$    &$8.5547e-03$&$1.96$ \\
$2^{8}$&$9.0650e-05$&$2.00$   &$3.1625e-04$&$2.00$    &$2.8566e-03$&$1.99$ \\
$2^{9}$&$3.5410e-05$&$2.00$   &$1.2353e-04$&$2.00$    &$1.2111e-03$&$2.00$ \\
$2^{10}$&$1.5738e-05$&$2.00$  &$5.4904e-05$&$2.00$    &$4.9827e-04$&$2.00$ \\
$2^{11}$&$7.7116e-06$&$-$     &$2.6903e-05$&$-$       &$2.4415e-04$&$-$ \\\hline
 $\varepsilon$&\multicolumn{2}{c}{$2^{-3}$}&\multicolumn{2}{c}{$2^{-5}$}&\multicolumn{2}{c}{$2^{-10}$}\\\hline\hline
    $N$ &$E_n$&Ord&$E_n$&Ord&$E_n$&Ord\\\hline
$2^{6}$   &$3.9901e-02$&$2.04$     &$3.9901e-02$&$2.04$  &$3.9901e-02$&$2.04$ \\
$2^{7}$   &$1.3288e-02$&$1.93$     &$1.3288e-02$&$1.93$  &$1.3288e-02$&$1.93$  \\
$2^{8}$   &$4.5122e-03$&$1.99$     &$4.5122e-03$&$1.99$  &$4.5122e-03$&$1.99$ \\
$2^{9}$   &$1.7709e-03$&$1.98$     &$1.7709e-03$&$1.98$  &$1.7709e-03$&$1.98$ \\
$2^{10}$  &$7.9347e-04$&$1.98$     &$7.9347e-04$&$1.98$  &$7.9347e-04$&$1.98$ \\
$2^{11}$  &$3.9158e-04$&$-$        &$3.9158e-04$&$-$     &$3.9158e-04$&$-$  \\\hline
 $\varepsilon$&\multicolumn{2}{c}{$2^{-15}$}&\multicolumn{2}{c}{$2^{-25}$}&\multicolumn{2}{c}{$2^{-30}$}\\\hline\hline
     $N$ &$E_n$&Ord&$E_n$&Ord&$E_n$&Ord\\\hline
$2^{6}$   &$3.9901e-02$&$2.04$     &$4.0243e-02$&$2.02$  &$4.0248e-02$&$2.02$ \\
$2^{7}$   &$1.3288e-02$&$1.93$     &$1.3581e-02$&$1.92$  &$1.3582e-02$&$1.92$  \\
$2^{8}$   &$4.5122e-03$&$1.99$     &$4.6375e-03$&$1.97$  &$4.6381e-03$&$1.97$ \\
$2^{9}$   &$1.7709e-03$&$1.98$     &$1.8372e-03$&$1.98$  &$1.8375e-03$&$1.98$ \\
$2^{10}$  &$1.7709e-03$&$1.98$     &$8.2321e-04$&$1.98$  &$8.2331e-04$&$1.98$\\
$2^{11}$  &$3.9158e-04$&$-$        &$4.0626e-04$&$-$     &$4.0631e-04$&$-$ \\\hline
 $\varepsilon$&\multicolumn{2}{c}{$2^{-35}$}&\multicolumn{2}{c}{$2^{-40}$}&\multicolumn{2}{c}{$2^{-45}$}\\\hline
\end{tabular}
\caption{Errors $E_N$ and convergence rates Ord for approximate solutions from Example \ref{primjer1}.}\label{tabela1}
\end{table}
\end{center}

\begin{example}\label{primjer2}\upshape
Consider the following problem 
  \begin{align}
      \varepsilon^2y''=y^3+y-2\text{ for } (0,1),\quad y(0)=y(1)=0,
   \label{primjer2a}
  \end{align}
whose exact solution is unknown. The nonlinear system was solved using the initial condition $y_0=1,$ that represents the reduced solution. The value of the constant $\gamma=4$ has been chosen so that the condition $\gamma\geqslant f_y(x,y),\:\forall (x,y)\in[0,1]\times [y_L,y_U]\subset[0,1]\times\mathbb{R}$ is fulfilled, where $y_L$ and $y_U$ are lower and upper solutions, respectively, of the problem \eqref{primjer2a}. Because of the fact that the exact solution is unknown, we are going to calculate $E_n$ using \eqref{greska2}.
 \label{primjer2}
\end{example}

\begin{center}
\begin{table}[!h]\small
\centering
\begin{tabular}{c|cc|cc|cc}\hline
     $N$ &$E_n$&Ord&$E_n$&Ord&$E_n$&Ord\\\hline
$2^{6}$  &$7.1345e-04$&$2.02$   &$3.7134e-03$&$2.01$    &$1.5182e-02$&$2.09$ \\
$2^{7}$  &$2.4017e-04$&$2.01$   &$1.2564e-04$&$2.01$    &$4.9236e-03$&$1.96$ \\
$2^{8}$  &$7.7985e-05$&$2.00$   &$3.1655e-04$&$2.00$    &$1.6403e-03$&$2.09$ \\
$2^{9}$  &$3.0463e-05$&$2.00$   &$1.2959e-04$&$2.00$    &$5.1903e-04$&$2.00$ \\
$2^{10}$ &$1.3539e-05$&$2.00$   &$3.9986e-05$&$2.00$    &$1.6001e-04$&$2.00$ \\
$2^{11}$ &$6.6341e-06$&$-$      &$1.2096e-05$&$-$       &$4.8389e-05$&$-$ \\\hline
 $\varepsilon$&\multicolumn{2}{c}{$2^{-3}$}&\multicolumn{2}{c}{$2^{-5}$}&\multicolumn{2}{c}{$2^{-10}$}\\\hline\hline
    $N$   &$E_n$&Ord&$E_n$&Ord&$E_n$&Ord\\\hline
$2^{6}$   &$1.5181e-02$&$2.09$     &$1.5181e-02$&$2.09$  &$1.5181e-02$&$2.09$ \\
$2^{7}$   &$4.9236e-03$&$1.96$     &$4.9236e-03$&$1.96$  &$4.9236e-03$&$1.96$  \\
$2^{8}$   &$1.6403e-03$&$2.00$     &$1.6403e-03$&$2.00$  &$1.6403e-03$&$2.00$  \\
$2^{9}$   &$5.1903e-04$&$2.00$     &$5.1903e-04$&$2.00$  &$5.1903e-04$&$2.00$ \\
$2^{10}$  &$1.6001e-04$&$2.00$     &$1.6001e-04$&$2.00$  &$1.6001e-04$&$2.00$\\
$2^{11}$  &$4.8389e-05$&$-$        &$4.8389e-05$&$-$     &$4.8389e-05$&$-$ \\\hline
 $\varepsilon$&\multicolumn{2}{c}{$2^{-15}$}&\multicolumn{2}{c}{$2^{-25}$}&\multicolumn{2}{c}{$2^{-30}$}\\\hline\hline
     $N$ &$E_n$&Ord&$E_n$&Ord&$E_n$&Ord\\\hline
$2^{6}$   &$1.5181e-02$&$2.09$     &$1.5184e-02$&$2.09$  &$1.5795e-02$&$2.09$ \\
$2^{7}$   &$4.9236e-03$&$1.96$     &$4.9221e-03$&$1.96$  &$5.1202e-03$&$1.96$  \\
$2^{8}$   &$1.6403e-03$&$2.00$     &$1.6436e-03$&$1.99$  &$1.7097e-03$&$1.99$ \\
$2^{9}$   &$5.1903e-04$&$2.00$     &$6.4509e-04$&$2.00$  &$6.7102e-04$&$2.00$ \\
$2^{10}$  &$1.6002e-04$&$2.00$     &$2.8669e-04$&$2.00$  &$2.9823e-04$&$2.00$\\
$2^{11}$  &$4.8390e-05$&$-$        &$1.4048e-04$&$-$     &$1.4613e-04$&$-$ \\\hline
 $\varepsilon$&\multicolumn{2}{c}{$2^{-35}$}&\multicolumn{2}{c}{$2^{-40}$}&\multicolumn{2}{c}{$2^{-45}$}\\\hline
\end{tabular}
\caption{Errors $E_N$ and convergence rates Ord for approximate solutions from Example \ref{primjer2}.}\label{tabela2}
\end{table}
\end{center}

\begin{example} \upshape
 Consider the following problem
\begin{equation*}
 \epsilon^2y''=y+1-2\varepsilon^2+x(x-1)\quad\text{for }x\in(0,1),\quad y(0)=y(1)=0.
\end{equation*}
The exact solution of this problem is given by
$\displaystyle
  y(x)=\frac{e^{-\frac{x}{\epsilon}}+e^{-\frac{1-x}{\epsilon}}}{1+e^{-\frac{1}{\epsilon}}}-x(x-1)-1.
$
The nonlinear system was solved using the initial condition $y_0=-0.5$ and the value of the constant  $\gamma=1$.
\label{primjer3}
\end{example}

\begin{example}\upshape
Consider the following problem 
  \begin{align}
      \varepsilon^2y''=y^3+y-2\text{ for } (0,1),\quad y(0)=y(1)=0,
   \label{primjer4}
  \end{align}
whose exact solution is unknown. The nonlinear system was solved using the initial condition $y_0=1,$ that represents the reduced solution. The value of the constant $\gamma=4$ has been chosen so that the condition $\gamma\geqslant f_y(x,y),\:\forall (x,y)\in[0,1]\times [y_L,y_U]\subset[0,1]\times\mathbb{R}$ is fulfilled, where $y_L$ and $y_U$ are lower and upper solutions, respectively, of the problem \eqref{primjer2a}. Because of the fact that the exact solution is unknown, we are going to calculate $E_n$ using \eqref{greska2}.
 \label{primjer4}
\end{example}

\newpage
\begin{center}
\begin{table}[!h]\small
\centering
\begin{tabular}{c|cc|cc|cc}\hline
     $N$ &$E_n$&Ord&$E_n$&Ord&$E_n$&Ord\\\hline
$2^{6}$&$9.0262e-04$&$2.05$   &$4.4799e-03$&$2.03$    &$3.9479e-02$&$2.01$ \\
$2^{7}$&$2.8729e-04$&$1.91$   &$1.4999e-03$&$1.92$    &$1.3362e-03$&$1.93$ \\
$2^{8}$&$9.8102e-05$&$1.95$   &$5.1221e-04$&$1.95$    &$4.5373e-03$&$1.96$ \\
$2^{9}$&$3.9049e-05$&$1.99$   &$2.0484e-04$&$1.99$    &$1.8060e-03$&$1.97$ \\
$2^{10}$&$1.7496e-05$&$1.99$  &$9.1409e-05$&$1.99$    &$8.1249e-04$&$1.97$ \\
$2^{11}$&$8.6345e-06$&$-$     &$4.4951e-05$&$-$       &$4.0241e-04$&$-$ \\\hline
 $\varepsilon$&\multicolumn{2}{c}{$2^{-3}$}&\multicolumn{2}{c}{$2^{-5}$}&\multicolumn{2}{c}{$2^{-10}$}\\\hline\hline
    $N$ &$E_n$&Ord&$E_n$&Ord&$E_n$&Ord\\\hline
$2^{6}$    &$3.9479e-02$&$2.01$  &$3.9479e-02$&$2.01$  &$3.9479e-02$&$2.01$ \\
$2^{7}$    &$1.3362e-03$&$1.93$  &$1.3362e-03$&$1.93$  &$1.3362e-03$&$1.93$  \\
$2^{8}$    &$4.5373e-03$&$1.96$  &$4.5373e-03$&$1.96$  &$4.5373e-03$&$1.96$ \\
$2^{9}$    &$1.8060e-03$&$1.97$  &$1.8060e-03$&$1.97$  &$1.8060e-03$&$1.97$ \\
$2^{10}$   &$8.1249e-04$&$1.97$  &$8.1249e-04$&$1.97$  &$8.1249e-04$&$1.97$ \\
$2^{11}$   &$4.0241e-04$&$-$     &$4.0241e-04$&$-$     &$4.0241e-04$&$-$   \\\hline
 $\varepsilon$&\multicolumn{2}{c}{$2^{-15}$}&\multicolumn{2}{c}{$2^{-25}$}&\multicolumn{2}{c}{$2^{-30}$}\\\hline\hline
     $N$ &$E_n$&Ord&$E_n$&Ord&$E_n$&Ord\\\hline
$2^{6}$    &$3.9479e-02$&$2.01$  &$3.9483e-02$&$2.01$   &$3.9485e-02$&$2.01$ \\
$2^{7}$    &$1.3362e-03$&$1.93$  &$1.3363e-03$&$1.93$   &$1.3364e-03$&$1.93$  \\
$2^{8}$    &$4.5373e-03$&$1.96$  &$4.5377e-03$&$1.95$   &$4.5378e-03$&$1.95$  \\
$2^{9}$    &$1.8060e-03$&$1.97$  &$1.8147e-03$&$1.97$   &$1.8180e-03$&$1.97$\\
$2^{10}$   &$8.1249e-04$&$1.97$  &$8.1641e-04$&$1.97$   &$8.1645e-04$&$1.97$\\
$2^{11}$   &$4.0241e-04$&$-$     &$4.0434e-04$&$-$      &$4.0436e-04$&$-$  \\\hline
 $\varepsilon$&\multicolumn{2}{c}{$2^{-35}$}&\multicolumn{2}{c}{$2^{-40}$}&\multicolumn{2}{c}{$2^{-45}$}\\\hline
\end{tabular}
\caption{Errors $E_N$ and convergence rates Ord for approximate solutions from Example \ref{primjer3}.}\label{tabela3}
\end{table}
\end{center}

\begin{center}
\begin{table}[!h]\small
\centering
\begin{tabular}{c|cc|cc|cc}\hline
     $N$ &$E_n$&Ord&$E_n$&Ord&$E_n$&Ord\\\hline
$2^{6}$  &$8.8623e-04$&$2.09$   &$3.4567e-03$&$2.11$    &$1.1656e-02$&$2.10$ \\
$2^{7}$  &$2.8728e-05$&$1.92$   &$1.1085e-03$&$1.93$    &$3.7537e-03$&$1.91$ \\
$2^{8}$  &$9.8102e-05$&$1.96$   &$3.7643e-04$&$1.95$    &$1.2923e-03$&$1.98$ \\
$2^{9}$  &$3.9049e-05$&$1.98$   &$1.5054e-04$&$1.98$    &$4.1404e-04$&$1.99$ \\
$2^{10}$ &$1.7496e-05$&$1.98$   &$6.7451e-05$&$1.99$    &$1.2855e-04$&$2.00$ \\
$2^{11}$ &$8.6345e-06$&$-$      &$3.3169e-05$&$-$       &$3.8914e-05$&$-$ \\\hline
 $\varepsilon$&\multicolumn{2}{c}{$2^{-3}$}&\multicolumn{2}{c}{$2^{-5}$}&\multicolumn{2}{c}{$2^{-10}$}\\\hline\hline
    $N$    &$E_n$&Ord&$E_n$&Ord&$E_n$&Ord\\\hline
$2^{6}$    &$1.1656e-02$&$2.10$  &$1.1656e-02$&$2.10$   &$1.1656e-02$&$2.10$ \\
$2^{7}$    &$3.7537e-03$&$1.91$  &$3.7537e-03$&$1.91$   &$3.7537e-03$&$1.91$  \\
$2^{8}$    &$1.2923e-03$&$1.98$  &$1.2923e-03$&$1.98$   &$1.2923e-03$&$1.98$  \\
$2^{9}$    &$4.1404e-04$&$1.99$  &$4.1404e-04$&$1.99$   &$4.1404e-04$&$1.99$ \\
$2^{10}$   &$1.2855e-04$&$2.00$  &$1.2855e-04$&$2.00$   &$1.2855e-04$&$2.00$\\
$2^{11}$   &$3.8914e-05$&$-$     &$3.8914e-05$&$-$      &$3.8914e-05$&$-$ \\\hline
 $\varepsilon$&\multicolumn{2}{c}{$2^{-15}$}&\multicolumn{2}{c}{$2^{-25}$}&\multicolumn{2}{c}{$2^{-30}$}\\\hline\hline
     $N$ &$E_n$&Ord&$E_n$&Ord&$E_n$&Ord\\\hline
$2^{6}$    &$1.1656e-02$&$2.10$  &$1.1656e-02$&$2.10$   &$1.1656e-02$&$2.10$ \\
$2^{7}$    &$3.7537e-03$&$1.91$  &$3.7537e-03$&$1.91$   &$3.7537e-03$&$1.91$  \\
$2^{8}$    &$1.2923e-03$&$1.98$  &$1.2923e-03$&$1.98$   &$1.2923e-03$&$1.98$ \\
$2^{9}$    &$4.1404e-04$&$1.99$  &$4.1404e-04$&$1.99$   &$4.1404e-04$&$1.99$ \\
$2^{10}$   &$1.2855e-04$&$2.00$  &$1.2855e-04$&$2.00$   &$1.2855e-04$&$2.00$\\
$2^{11}$   &$3.8914e-05$&$-$     &$3.8914e-05$&$-$      &$3.8914e-05$&$-$\\\hline
 $\varepsilon$&\multicolumn{2}{c}{$2^{-35}$}&\multicolumn{2}{c}{$2^{-40}$}&\multicolumn{2}{c}{$2^{-45}$}\\\hline
\end{tabular}
\caption{Errors $E_N$ and convergence rates Ord for approximate solutions from Example \ref{primjer4}.}\label{tabela4}
\end{table}
\end{center}


\newpage

\begin{figure}[!h]
     \label{grfaik4}
     \begin{subfigure}[b]{.5\textwidth}\centering
        \includegraphics[scale=.45]{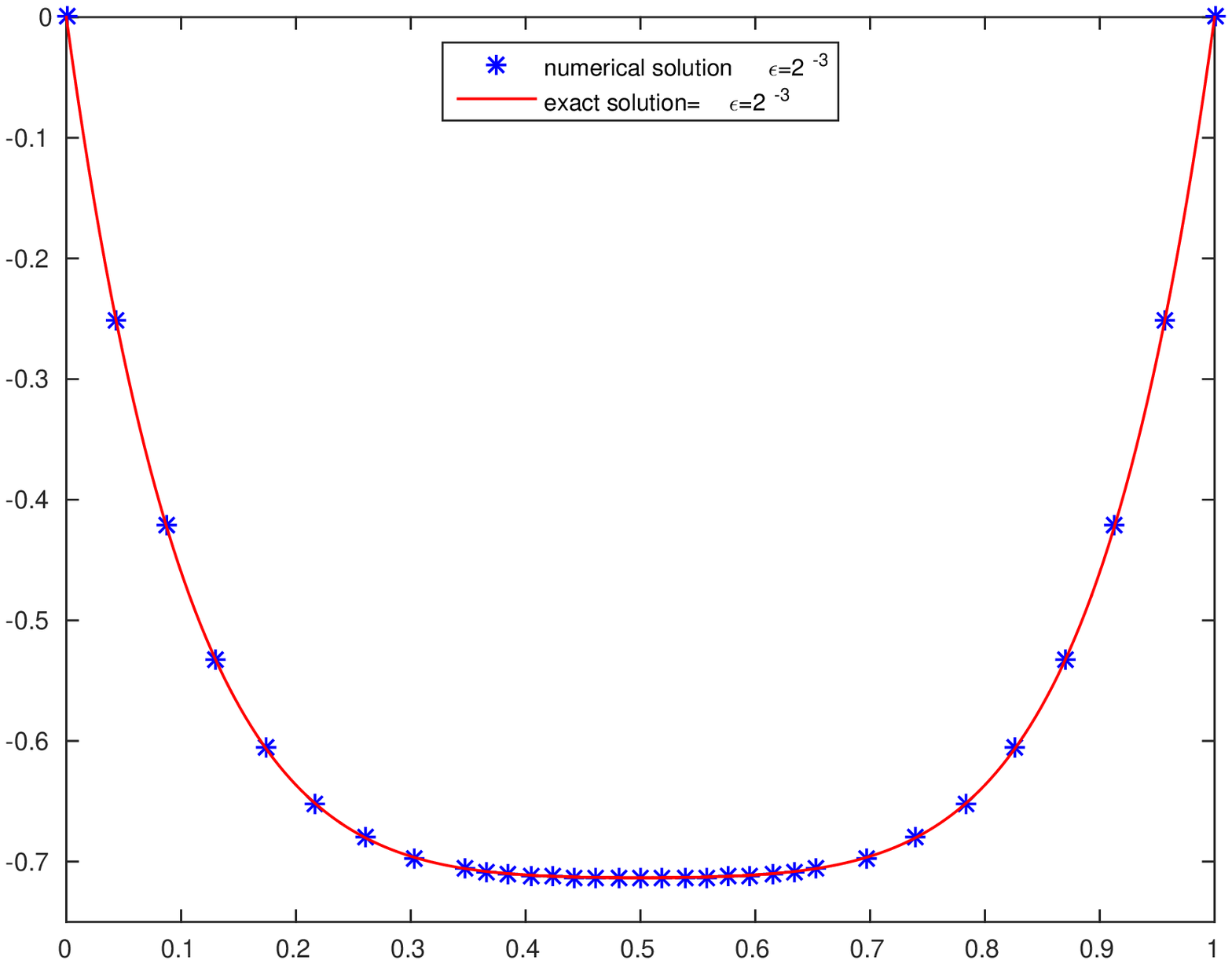}
         \caption{$\varepsilon=2^{-3},\:N=32$}
          \label{grqfik1}
   \end{subfigure}
    \begin{subfigure}[b]{.5\textwidth}\centering
        \includegraphics[scale=.45]{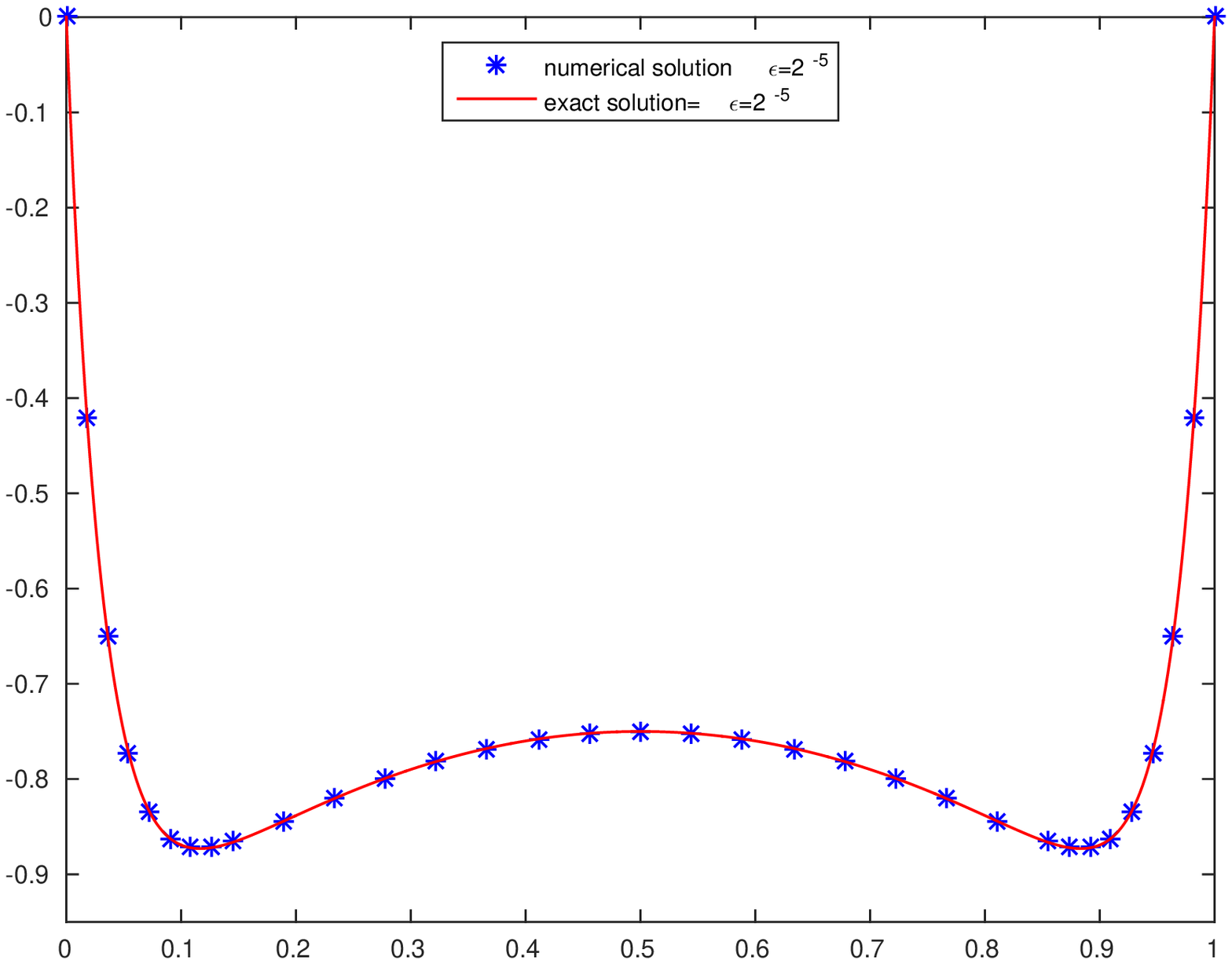}
         \caption{$\varepsilon=2^{-5},\:N=32$}
      \label{grqfik2}   
     \end{subfigure}
     \\
     \begin{subfigure}[b]{.5\textwidth}\centering
        \includegraphics[scale=.45]{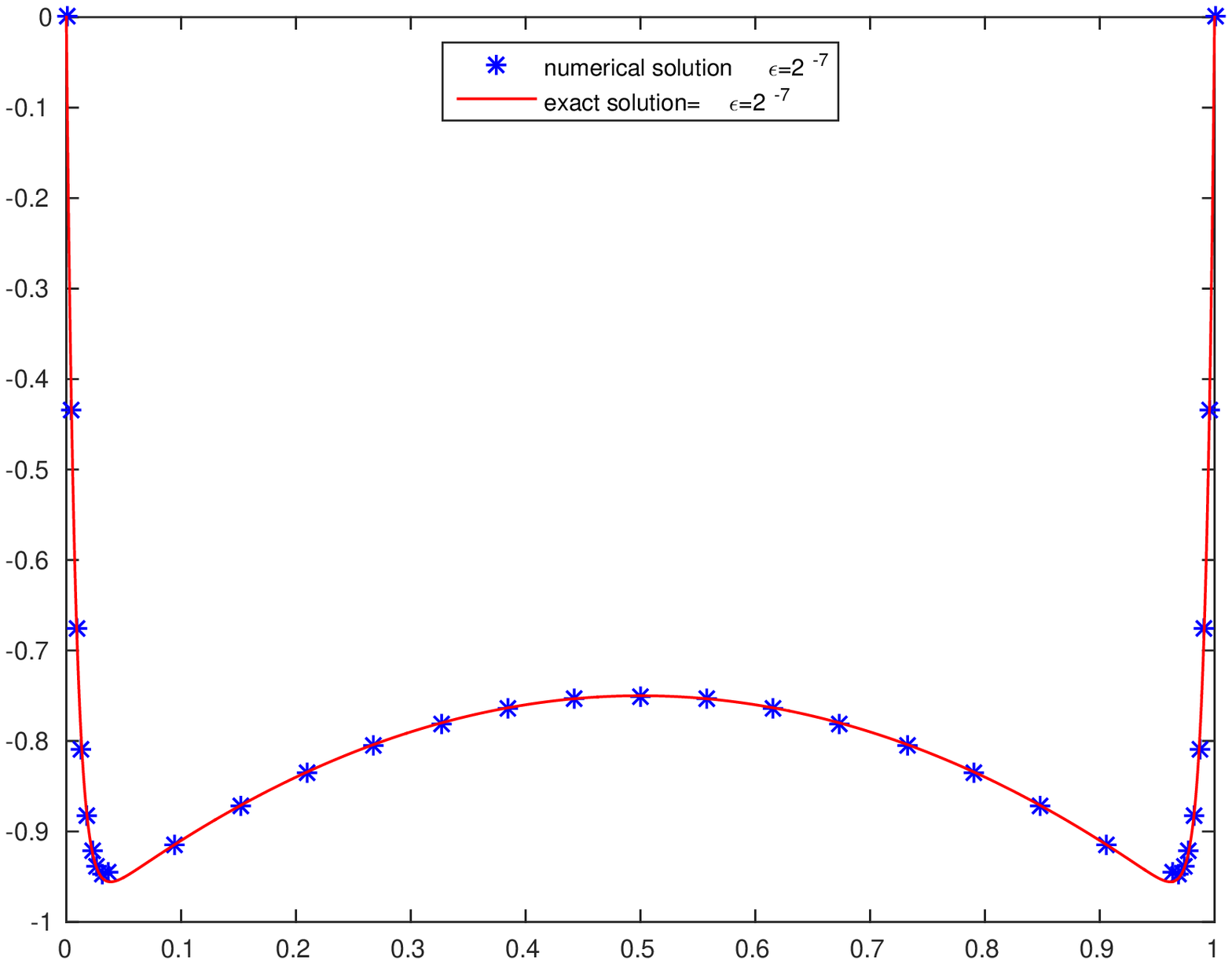}
         \caption{$	\varepsilon=2^{-7},\:N=32$}
      \label{grafik3} 
      \end{subfigure}
       \begin{subfigure}[b]{.5\textwidth}\centering
        \includegraphics[scale=.45]{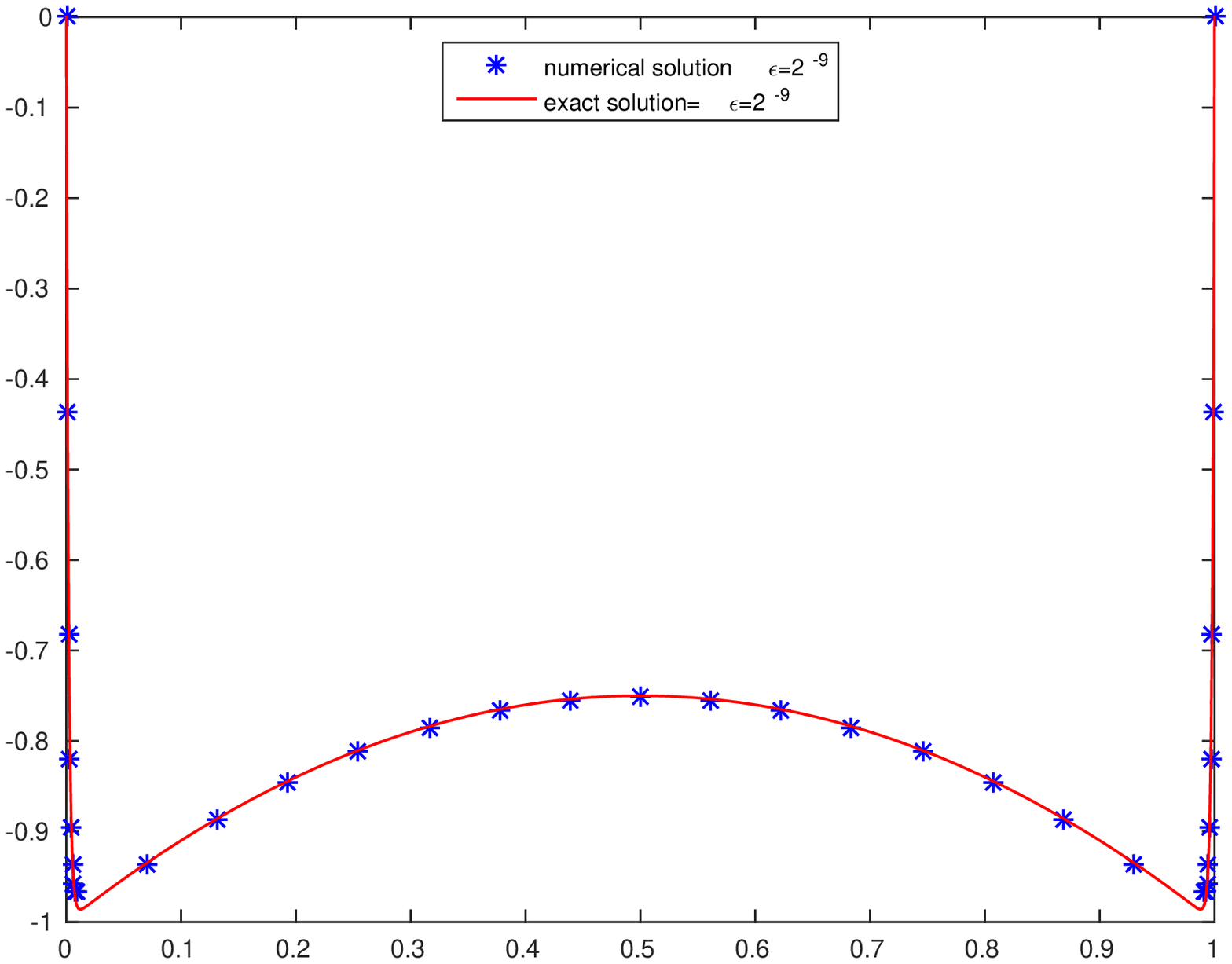}
         \caption{$\varepsilon=2^{-9},\:N=32$}
          \label{grafik4}
   \end{subfigure}
        \caption{Graphics of the numerical and exact solutions for $N=32$ and $\varepsilon=2^{-3},\,2^{-5},\,2^{-7},\,2^{-9}$ 
                    for Example \ref{primjer1} and Example \ref{primjer3}  }
  \end{figure}

\begin{figure}[!h]
     \begin{subfigure}[b]{.5\textwidth}\centering
        \includegraphics[scale=.45]{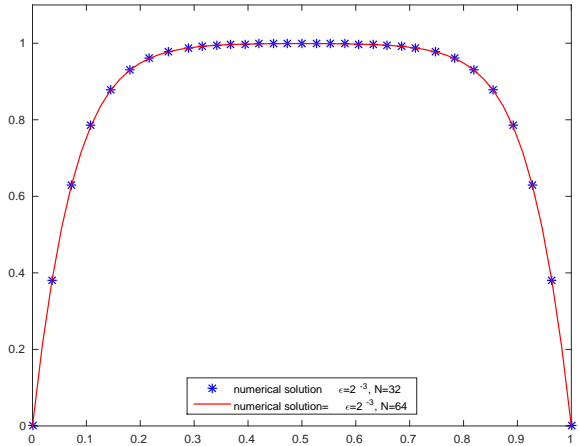}
         \caption{$\varepsilon=2^{-3},\:N=32$}
          \label{grqfik5}
   \end{subfigure}
    \begin{subfigure}[b]{.5\textwidth}\centering
        \includegraphics[scale=.45]{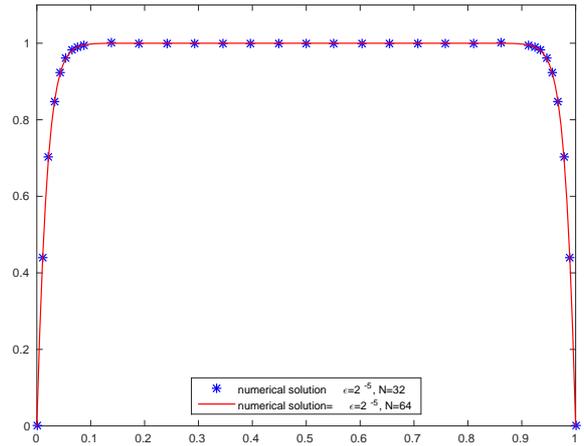}
         \caption{$\varepsilon=2^{-5},\:N=32$}
      \label{grqfik6}   
     \end{subfigure}
     \\
     \begin{subfigure}[b]{.5\textwidth}\centering
        \includegraphics[scale=.45]{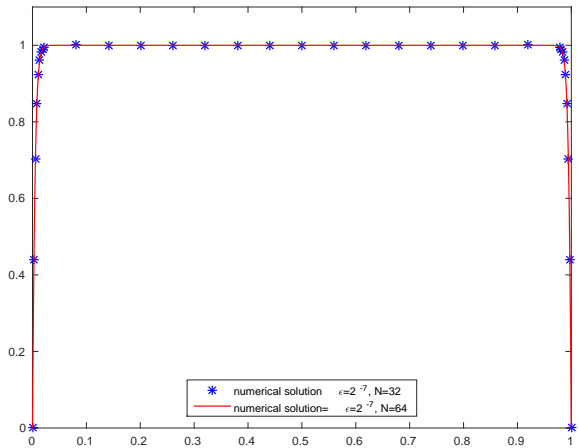}
         \caption{$	\varepsilon=2^{-7},\:N=32$}
      \label{grafik7} 
      \end{subfigure}
       \begin{subfigure}[b]{.5\textwidth}\centering
        \includegraphics[scale=.45]{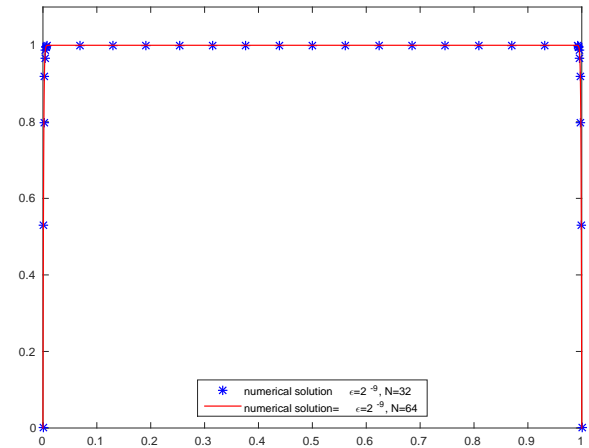}
         \caption{$\varepsilon=2^{-9},\:N=32$}
          \label{grafik8}
   \end{subfigure}
        \caption{Graphics of the numerical and solutions for $N=32,\,64$ and $\varepsilon=2^{-3},\,2^{-5},\,2^{-7},\,2^{-9}$  for Example \ref{primjer2} and Example \ref{primjer4} } 
     \label{grafik4}
  \end{figure}

%

%
\newpage


\noindent Department of Mathematics\\
Faculty of sciences and Mathematics,\\
University of Tuzla\\
Univerzitetska 4, 75000 Tuzla,\\
Bosnia and Herzegovina\\
E-mail address:\texttt{  samir.karasuljic@untz.ba}\\\\

\noindent Department of Mathematics\\
Faculty of sciences and Mathematics,\\
University of Tuzla\\
Univerzitetska 4, 75000 Tuzla,\\
Bosnia and Herzegovina\\
E-mail address: \texttt{enes.duvnjakovic@untz.ba}\\\\

\noindent Department of Mathematics\\
Faculty of sciences and Mathematics,\\
University of Tuzla\\
Univerzitetska 4, 75000 Tuzla,\\
Bosnia and Herzegovina\\
E-mail address: \texttt{memic{\_}91{\_}elvir@hotmail.com}
\end{document}